\RequirePackage{fix-cm}
\documentclass[smallextended]{svjour3}       
\smartqed  
\usepackage{float}
\usepackage{cite}%
\usepackage{color}
\usepackage{graphicx}
\usepackage{amscd, amsmath, amssymb}
\usepackage{pgf,tikz}
\usetikzlibrary{arrows}
\usepackage{pgfplots}
\usepackage{enumerate}
\usepackage{caption}
\usepackage{subcaption}
\usepackage{tikz}
\usepackage{pgfplots}
\usepackage{adjustbox}
\usepackage{hyperref}
\usepackage{todonotes}
\usepackage{amsmath}
\usepackage{amssymb}

\usepackage{amsthm}
\usepackage{amsfonts}
\usepackage{enumerate}
\usepackage{mathtools} 
\usepackage{xcolor}
\usepackage{rotating,booktabs,multirow}
\usepackage{varwidth}
\usepackage{algorithm,algpseudocode}
\algnewcommand{\Inputs}[1]{%
  \State \textbf{Inputs:}
  \Statex \hspace*{\algorithmicindent}\parbox[t]{.8\linewidth}{\raggedright #1}
}
\algnewcommand{\Initialize}[1]{%
  \State \textbf{Initialization:}
  \Statex \hspace*{\algorithmicindent}\parbox[t]{.8\linewidth}{\raggedright #1}
}

\DeclareMathOperator{\st}{s.t.}
\DeclareMathOperator{\tr}{tr}
\usepackage{algpseudocode}

\algnewcommand{\algorithmicand}{\textbf{ and }}
\algnewcommand{\algorithmicor}{\textbf{ or }}
\algnewcommand{\OR}{\algorithmicor}
\algnewcommand{\AND}{\algorithmicand}
\algnewcommand{\var}{\texttt}

\begin{document}

\title{The exact worst-case convergence rate of the gradient method with fixed step lengths for $L$-smooth functions\thanks{This work was supported by the
 Dutch Scientific Council (NWO)  grant OCENW.GROOT.2019.015, \emph{Optimization for and with Machine Learning (OPTIMAL)}.}
}

\titlerunning{Convergence rate of gradient method}        

\author{ Hadi Abbaszadehpeivasti
  \and
            Etienne de Klerk
            \and Moslem Zamani\footnote{Corresponding author.}}


\institute{
H. Abbaszadehpeivasti\at Tilburg University, Department of Econometrics and Operations Research, Tilburg, The Netherlands\\
\email{h.abbaszadehpeivasti@tilburguniversity.edu}
\and
E. De Klerk\at Tilburg University, Department of Econometrics and Operations Research, Tilburg, The Netherlands\\
\email{e.deklerk@tilburguniversity.edu}
\and
M. Zamani\at Tilburg University, Department of Econometrics and Operations Research, Tilburg, The Netherlands\\
\email{m.zamani\_1@tilburguniversity.edu}
}

\date{Received: date / Accepted: date}

\maketitle

\begin{abstract}
In this paper, we study the convergence rate of
the gradient (or steepest descent) method with fixed step lengths for finding a stationary point of an $L$-smooth function.
 We establish a new convergence rate, and show that the bound may be exact in some cases, in particular when all step lengths lie in the interval $(0,1/L]$. In addition,  we derive an optimal step length with respect to the new bound.
\end{abstract}

\section{Introduction}
\label{intro}

We consider the non-convex unconstrained optimization problem
\begin{align}\label{P}
\min_{x\in \mathbb{R}^n} f(x),
\end{align}
where $f: \mathbb{R}^n\to \mathbb{R}$ is bounded from below, and let a real number $f^\star$ denote a lower bound of problem \eqref{P}. In addition, we assume throughout the paper that $f$ has an $L$-Lipschitz gradient, that is
\begin{align}\label{D2}
\| \nabla f(x)-\nabla f(y)\|\leq L \|x-y\|  \ \ \ \forall x, y\in \mathbb{R}^n,
\end{align}
for some (known) Lipschitz constant $L >0$. Following the notation used by Nesterov \cite{Nesterov}, we let $C^{1,1}_L(\mathbb{R}^n)$ denote functions with $L$-Lipschitz gradient.

 Problem \eqref{P}  arises naturally in many applications including machine learning, signal and image processing, to name but a few \cite{bottou2018optimization, MAL-058}. One of the historic solution methods for problem \eqref{P} is the  gradient  method,  proposed by Cauchy in 1847 \cite{Cauchy}.

The gradient method with fixed step lengths may be described as follows.
\begin{algorithm}
\caption{Gradient method  with fixed step lengths}
\begin{algorithmic}
\State Set $N$ and $\{t_k\}_{k=1}^N$ (step lengths) and pick $x^1\in\mathbb{R}^n$.
\State For $k=1, 2, \ldots, N$ perform the following step:\\
\begin{enumerate}
\item
$x^{k+1}=x^k-t_k\nabla f(x^k)$
\end{enumerate}
\end{algorithmic}
\label{Alg1}
\end{algorithm}

Nesterov {{\cite[page 28]{Nesterov}}} gives the following convergence rate (to a stationary point) for Algorithm \ref{Alg1} when $t_k\in (0, \tfrac{2}{L})$, $k\in\{1, \ldots, N\}$:
$$
\min_{1\leq k\leq N+1} \left\|\nabla f(x^k)\right\|\leq \left(\frac{f(x^1)-f^\star}{\left(\sum_{k=1}^N t_k(1-\tfrac{1}{2}Lt_k)\right)+\tfrac{1}{2L}}\right)^{\tfrac{1}{2}}.
$$
In the special case $t_k=\tfrac{1}{L}$, $k\in\{1, \ldots, N\}$, the last bound becomes
$$
\min_{1\leq k\leq N+1} \left\|\nabla f(x^k)\right\|\leq \left(\frac{2L(f(x^1)-f^\star)}{N+1}\right)^{\tfrac{1}{2}}.
$$

Recently, semidefinite programming performance estimation have been employed as a tool for the worst-case analysis
 of first-order methods \cite{drori2014performance, taylor2017smooth, Taylor, de2017worst, de2020worst}. In this method,
  the worst-case convergence is cast as a quadratic program with quadratic constraints and the problem is then solved by semi-definite programming methods.
By employing the performance estimation method, Taylor {{\cite[page 190]{Taylot_T}}}, without giving a proof, states the following convergence rate:
\begin{equation}
\label{claim:Taylor}
\min_{1\leq k\leq N+1}\left \|\nabla f(x^k)\right\|\leq \left(\frac{4L(f(x^1)-f^\star)}{3N}\right)^{\tfrac{1}{2}},
\end{equation}
for $t_k=\tfrac{1}{L}$, $k\in\{1, \ldots, N\}$. Drori {{\cite[Corollary 1 in Appendix]{Drori}}} considers the case
 that all step lengths are smaller than $\tfrac{1}{L}$, and proves  the following convergence rate
 \begin{align}\label{Dr_R}
\min_{1\leq k\leq N+1} \left\|\nabla f(x^k)\right\|\leq \left(\frac{4(f(x^1)-f^\star)}{\sum_{k=1}^N t_k(4-Lt_k)}\right)^{\tfrac{1}{2}}.
 \end{align}
 It can be observed that when the step lengths are the same for each iteration and tend to $\tfrac{1}{L}$, the bound \eqref{Dr_R} reduces to
  Taylor's convergence rate.

In this paper, we investigate the convergence rate of Algorithm \ref{Alg1} further.
By using the performance estimation method, we provide a converge rate, which is tighter than all aforementioned bounds.
For example, as a part of our main result in Theorem \ref{thm:main}, we improve on \eqref{Dr_R} by showing, for any choice of $t_k \in (0,\sqrt{3}/L)$ ($k\in\{1, \ldots, N\}$), that
\begin{equation}
\label{bound:Thm2}
 \min_{1\leq k\leq N+1}\left\|\nabla f(x^k)\right\|\leq \left(\frac{4\Delta}{\sum_{k=1}^N \min(-L^2t_k^3+4t_k,-Lt_k^2+4t_k)+\tfrac{2}{L}}\right)^{\tfrac{1}{2}}.
 \end{equation}
As a consequence, we also prove and improve on \eqref{claim:Taylor} by showing, in the special case where all $t_k = 1/L$ ($k\in\{1, \ldots, N\}$), that
\[
\min_{1\leq k\leq N+1} \left\|\nabla f(x^k)\right\|\leq \left(\tfrac{4L(f(x^1)-f^\star)}{3N+2}\right)^{\tfrac{1}{2}}.
\]
 In addition, we construct an $L$-smooth function that attains the given bound in Theorem \ref{thm:main} for certain step lengths. We also propose an optimal
  step length that minimizes the right-hand-side of the bound \eqref{bound:Thm2}, namely $t_k=\tfrac{\sqrt{4/3}}{L}$ for all $k \in \{1,\ldots,N\}$.

\subsubsection*{Outline}
The paper is organized as follows. We describe the performance estimation technique in Section \ref{sec:performance estimation}.
In Section \ref{sec:convergence rate}, we study the convergence rate by using performance estimation. Finally,  we conclude the paper with a conjecture.

\subsubsection*{Notation} The $n$-dimensional Euclidean space is denoted by $\mathbb{R}^n$.
 We use $\langle \cdot, \cdot\rangle$ and $\| \cdot\|$ to denote the Euclidean inner product and norm, respectively.
  For a matrix $A$, $A_{ij}$ denotes its $(i, j)$-th entry,
  and $A^T$ represents the transpose of $A$. The  notation $A\succeq 0$ means the matrix $A$ is symmetric positive semi-definite.

\section{Performance estimation}
\label{sec:performance estimation}
Computation of the worst-case convergence rate for a given iterative method and  a given class of functions is
an infinite-dimensional optimization problem. In their seminal paper \cite{drori2014performance}, Drori and Teboulle take advantage of this idea,
 called performance estimation, and introduce  some relaxation method to deal with this infinite-dimensional optimization problem.
 Performance estimation has been used extensively  for the analysis of first-order
  methods \cite{drori2014performance, taylor2017smooth, Taylor, de2017worst, de2020worst}.

 Similar to problem (P) in \cite{drori2014performance}, the worst-case convergence rate of Algorithm \ref{Alg1} may be formulated as the following abstract optimization problem,
\begin{align}\label{P1}
\nonumber   \max & \ \left(\min_{1\leq k\leq N+1} \left\|\nabla f(x^k)\right\|\right)\\
\nonumber  \st \ &  f(x^1)-f^\star\leq \Delta  \\
&  \  x^{N+1}, x^N, \ldots., x^2 \ \textrm{are generated  by Algorithm \ref{Alg1} w.r.t.}\ f, x^1  \\
\nonumber  & \ f(x)\geq f^\star \ \forall x\in\mathbb{R}^n\\
\nonumber  & \ f\in C^{1,1}_L(\mathbb{R}^n)\\
\nonumber  & \ x^1\in\mathbb{R}^n,
\end{align}
where $\Delta\geq0$ denote the difference between the given lower bound, $f^\star$, and the value of $f$ at the starting point. In problem
 \eqref{P1}, $f$ and $x^1$ are decision variables. This is an infinite-dimensional optimization problem with infinite number of constraints,
  and consequently intractable in general. In what follows, we provide a semidefinite programming relaxation for the problem.
\begin{definition}{{\cite[Definition 3.8.]{Taylor}}}
Let $L\geq 0$. A differentiable function $f : \mathbb{R}^n\to\mathbb{R}$ is called
$L$-smooth, if it satisfies the following condition,
\begin{align}\label{D1}
\left|f(y)-f(x)-\left\langle \nabla f(x), y-x\right\rangle\right|\leq \tfrac{L}{2}\left\|y-x\right\|^2 \ \ \forall x, y\in \mathbb{R}^n.
\end{align}
\end{definition}
 The following proposition states a well-known characterization of $L$-smooth functions that follows, e.g., from \cite[Lemma 1.2.3]{Nesterov}, \cite[ Theorem 2.1.5]{Nesterov} and  \cite[Lemma 3.9]{Taylor}.

\begin{proposition}
Let $L\geq 0$. $f : \mathbb{R}^n\to\mathbb{R}$ is $L$-smooth, if and only if it has an $L$-Lipschitz gradient.
\end{proposition}

The following well-known result is a fundamental property of gradient descent for $L$-smooth functions, if the step length $1/L$ is used.
\begin{proposition}{{\cite[page 26]{Nesterov}}}
\label{lemma:descent}
If $f : \mathbb{R}^n\to\mathbb{R}$ is $L$-smooth, and $x \in \mathbb{R}^n$, then
\[
f\left(x - \frac{1}{L}\nabla f(x) \right) \le f(x) - \frac{1}{2L}\|\nabla f(x)\|^2.
\]
\end{proposition}

The following theorem plays a key role in our analysis. Indeed, it provides necessary and sufficient conditions for the
 interpolation of $L$-smooth functions. Using this theorem, we will formulate problem \eqref{P1} as a finite dimensional optimization problem.
\begin{theorem}[{{\cite[Lemma 3.9.]{Taylor}, \cite[Theorem 7 in Appendix]{Drori}}}]\label{T1}
Let  $\{(x^i; g^i; f^i)\}_{i\in I}\subseteq \mathbb{R}^n\times \mathbb{R}^n \times \mathbb{R}$ with a given index set $I$ and $L > 0$. There
 exists an $L$-smooth function  $f$ with
\begin{align}\label{int_fg}
f(x^i) = f^i, \nabla f(x^i) = g^i \ \ i\in I,
\end{align}
if and only if
{\small{
\begin{align}\label{Int-c}
\tfrac{1}{2L}\left\|g^i-g^j\right\|^2-\tfrac{L}{4}\left\|x^i-x^j-\tfrac{1}{L}(g^i-g^j)\right\|^2\leq f^i-f^j-\left\langle g^j, x^i-x^j\right\rangle \ \ i, j\in I.
\end{align}
}}
In addition, if the triple $\{(x^i; g^i; f^i)\}_{i\in I}$ satisfies \eqref{Int-c}, then there exists $L$-smooth function $f$ for which \eqref{int_fg} holds
and $\min_{x\in\mathbb{R}^n} f(x)=\min_{i\in I} f_i-\tfrac{1}{2L}\|g^i\|^2$. Moreover, letting $i^* \in \arg\min_{i\in I} f_i-\tfrac{1}{2L}\|g^i\|^2$,
a global minimizer of this function is given by $x^\star = x_{i^*} - \frac{1}{L}g^{i^*}$.
\end{theorem}
Another proof of the first part of Theorem \ref{T1} may be found in  \cite[Theorem 2, Page 148]{Wells}. By virtue of Theorem \ref{T1}, problem \eqref{P1} may be reformulated as follows,
\begin{align}\label{P2}
\nonumber   \max & \ \left(\min_{1\leq k\leq N+1} \left\|g^k\right\|\right)\\
\nonumber \st \ & \tfrac{1}{2L}\left\|g^i-g^j\right\|^2-\tfrac{L}{4}\left\|x^i-x^j-\tfrac{1}{L}(g^i-g^j)\right\|^2\leq f^i-f^j-\\
\nonumber & \ \ \ \ \ \left\langle g^j, x^i-x^j\right\rangle \ \ i, j\in\left\{1, \ldots, N+1\right\}  \\
&  \  x^{k+1}=x^k-t_k g^k  \ \ k\in\left\{1, \ldots, N\right\} \\
\nonumber& \ f^k\geq f^\star  \ \ k\in\left\{1, \ldots, N+1\right\}\\
\nonumber  &  f^1-f^\star\leq \Delta.
\end{align}
In the above formulation, $x^k,\ g^k,\ f^k$, $k\in \left\{1, \ldots, N+1\right\}$,  are decision variables. Note that in the above formulation, the constraints $f(x)\geq f^\star$ for each $x\in\mathbb{R}^n$ are replaced by $f^k\geq f^\star,  \ \ k\in\left\{1, \ldots, N+1\right\}$. These constraints do not necessarily impose a given $L$-Lipschitz  function $f$ with
$$
f(x^i) = f^i, \nabla f(x^i) = g^i \ \ i\in\left\{1, \ldots, N+1\right\},
$$
which  has the lower bound $f^\star$.
 Therefore, the optimal value of \eqref{P1} and \eqref{P2} may not be equal in general. However, if an optimal solution of problem \eqref{P2} satisfies $f^\star=\min_{1\leq k\leq N+1} f^k-\tfrac{1}{2L}\|g^k\|^2$, the formulation will be exact; see the second part of Theorem \ref{T1}. By Proposition \ref{lemma:descent}, we have $f(x)-\tfrac{1}{2L}\|\nabla f(x)\|^2\geq f^\star$ for $x\in\mathbb{R}^n$. Hence, we replace the constraint $f^k\geq f^\star$ by $f^k-\tfrac{1}{2L}\|g^k\|^2\geq f^\star$ and  consider the following problem:
\begin{align}\label{P3}
\nonumber   \max & \ \left(\min_{1\leq k\leq N+1} \left\|g^k\right\|\right)\\
\nonumber \st \ & \tfrac{1}{2L}\left\|g^i-g^j\right\|^2-\tfrac{L}{4}\left\|x^i-x^j-\tfrac{1}{L}(g^i-g^j)\right\|^2\leq f^i-f^j-\\
\nonumber & \ \ \ \ \ \left\langle g^j, x^i-x^j\right\rangle \ \ i, j\in\left\{1, \ldots, N,N+1\right\}  \\
&  \  x^{k+1}=x^k-t_k g^k  \ \ k\in\left\{1, \ldots, N+1\right\} \\
\nonumber& \ f^k-\tfrac{1}{2L}\|g^k\|^2- f^\star\geq 0  \ \ k\in\left\{1, \ldots, N+1\right\}\\
\nonumber  &  f^1-f^\star\leq \Delta.
\end{align}

From the constraint $x^{k+1}=x^k-t_k g^k$, we get $x^i=x^1+\sum_{k=1}^{i-1} g^k$, $i\in\left\{2, \ldots, N\right\}$.
By using this relation to eliminate the $x^i$
($i\in\left\{2, \ldots, N+2\right\}$), problem \eqref{P3} may be written as follows:
{\small{
\begin{align}\label{P4}
 \nonumber  \max & \ \ell\\
 \nonumber  \st \ &  f^i-f^j-\tfrac{1}{2L}\left\|g^i-g^j\right\|^2+ \tfrac{L}{4}\left\|-\sum_{k=j}^{i-1} t_kg^k+\tfrac{1}{L}(g^i-g^j)\right\|^2+
 \left\langle g^j, \sum_{k=j}^{i-1} t_kg^k\right\rangle\geq 0 \ \ i> j  \\
 \nonumber  &  f^i-f^j-\tfrac{1}{2L}\left\|g^i-g^j\right\|^2 +\tfrac{L}{4}\left\|\sum_{k=i}^{j-1} t_kg^k-\tfrac{1}{L}(g^i-g^j)\right\|^2-
 \left\langle g^j, \sum_{k=i}^{j-1} t_kg^k\right\rangle\geq 0 \ \ i< j  \\
  & \ f^k-\tfrac{1}{2L}\|g^k\|^2- f^\star\geq 0  \ \ k\in\left\{1, \ldots, N+1\right\}\\
  \nonumber  &  f^\star-f^1+\Delta\geq 0\\
  \nonumber  &  \left\| g^k\right\|^2-\ell\geq 0  \ \ \ k\in\left\{1, \ldots, N+1\right\},
\end{align}
}}
\hspace{-.245cm} where $\ell$ is an auxiliary variable  to convert problem \eqref{P3} into a quadratic program. Problem \eqref{P4} is a non-convex quadratic program with quadratic constraints. In the following proposition, we show that the optimal values of problems \eqref{P1} and \eqref{P3} (or equivalently problem \eqref{P4}) are the same for step lengths in the interval $(0, \tfrac{1}{2L})$.
\begin{proposition}
If $t_k\in(0, \tfrac{2}{L})$, $k\in\{1, \dots, N\}$, then    problems \eqref{P1} and \eqref{P3} (or equivalently problem \eqref{P4}) share the same optimal value.
\end{proposition}
\begin{proof}
Clearly,  problem \eqref{P3} is a relaxation of  problem \eqref{P1}. Therefore, we only need to show that, for any feasible solution of \eqref{P3}, say $\{(\bar x^i; \bar g^i; \bar f^i)\}_{1}^{N+1}$, there exists an $L$-smooth function $f$ with
 $$
 f(\bar x^i)=\bar f^i, \ \nabla f(\bar x^i)=\bar g^i, \ \ \    1\leq i\leq N+1,
 $$
and $\min_{x\in\mathbb{R}^n} f(x)\geq f^\star$. The existence such of a function follows from Theorem \ref{T1}, as  all assumptions of Theorem \ref{T1} are satisfied.
\end{proof}

To obtain a tractable form of problem \eqref{P4}, we relax it to a semidefinite program, in the spirit of \cite{drori2014performance}. To this end, we define the $(N+1)\times(N+1)$  positive semi-definite matrix $G$ as,
\begin{align*}
G=&
\begin{pmatrix}
 \left(g^1\right)^T\\ \vdots \\ \left(g^{N+1}\right)^T
\end{pmatrix}
\begin{pmatrix}
 g^1& \hdots & g^{N+1} \end{pmatrix}=
\begin{pmatrix}
 \left\|g^1\right\|^2  &~ \hdots &~  \left\langle g^1, g^{N+1}\right\rangle\\
\vdots  &~ \ddots &~\vdots \\
 \left\langle g^1, g^{N+1}\right\rangle  &~ \hdots &~   \left\|g^{N+1}\right\|^2
\end{pmatrix}.
\end{align*}
We may now formulate the following semidefinite program,
\begin{align}\label{S1}
 \nonumber   \max & \  \ell\\
 \nonumber \st \ &  f^i-f^j+\tr(A^{ij}G)\geq 0 \ \ \ \ i\neq j\in\{1, \ldots, N+1\}  \\
  \nonumber &  f^k- \tfrac{1}{2L}G_{kk}-f^\star\geq 0  \ \ \ \ k\in\{1, \ldots, N+1\}\\
  &  f^\star-f^1+\Delta\geq 0\\
  \nonumber  & G_{kk}-\ell\geq 0 \ \ \ \ k\in\{1, \ldots, N+1\}\\
 \nonumber& G\succeq 0,
\end{align}
where the $(N+1)\times(N+1)$ matrices $A^{ij}$, $i\neq j\in\{1, \ldots, N+1\}$, are formed
according to the constraints  \eqref{P4}, and $G, \ell, f^i$, $i\in\{1, \ldots, N+1\}$, are decision variables.
 Problem \eqref{S1} is a relaxation of \eqref{P4}, but if $n\geq N+1$ the relaxation is exact,
  that is  the optimal values of \eqref{P4} and \eqref{S1} are the same. Indeed, if $n\geq N+1$, and $G$ is a feasible matrix in \eqref{S1},
  then $G$ is the Gram matrix of $N+1$ vectors in $\mathbb{R}^n$, and these vectors may be identified with $g^1, \ldots, g^{N+1}$;
    a similar argument is used in
   \cite[Theorem 5]{taylor2017smooth}.
\section{Worst-case convergence rate}
\label{sec:convergence rate}
In this section, we investigate the convergence rate of gradient method with fixed step lengths.
The next theorem gives the worst-case convergence rate of Algorithm \ref{Alg1} to a stationary point of an $L$-smooth function.
The technique of the proof, as is usual for SDP performance estimation, is to use weak duality.
In particular, we will in fact construct a feasible solution to the dual SDP problem of \eqref{S1}, and thus derive an upper bound for problem  \eqref{P4}.

In practice, this dual feasible solution is constructed in a computer-assisted manner, by solving the primal and dual SDP problems
for different fixed values of the parameters, and subsequently guessing the values of the dual multipliers.
(There is also dedicated software for this purpose, namely `PESTO' by Taylor, Glineur, and Hendrickx \cite{8263832}.)
In the proof
of Theorem \ref{thm:main}, we simply verify that these `guesses' are correct.

 \begin{theorem}
 \label{thm:main}
Let $t_k\in(0, \tfrac{\sqrt{3}}{L})$ for $k\in\{1, \ldots, N\}$.
Consider  $N$ iterations of Algorithm \ref{Alg1} with step lengths $t_k$ ($k\in\{1, \ldots, N\}$), applied to some $L$-smooth function $f$ with
minimum value $f^\star$, with the starting point $x^1$ satisfying $f(x^1) - f^\star \le \Delta$, for some given $\Delta >0$.

Then, if $x^1,\ldots,x^{N+1}$ denote the iterates of Algorithm \ref{Alg1}, one has
\begin{align}\label{B1}
 \min_{1\leq k\leq N+1}\left\|\nabla f(x^k)\right\|\leq \left(\frac{4\Delta}{\sum_{k=1}^N \min(-L^2t_k^3+4t_k,-Lt_k^2+4t_k)+\tfrac{2}{L}}\right)^{\tfrac{1}{2}}.
\end{align}
In particular, if $t_k=\tfrac{\sqrt{4/3}}{L}$ for $k\in\left\{1, \ldots, N\right\}$, we get
\begin{align}\label{B3}
\min_{1\leq k\leq N+1} \left\|\nabla f(x^k)\right\|\leq \left(\tfrac{6\sqrt{3}L(f(x^1)-f^\star)}{8N+3\sqrt{3}}\right)^{\tfrac{1}{2}}.
\end{align}
Similarly, if $t_k=\tfrac{1}{L}$ for $k\in\left\{1, \ldots, N\right\}$, one has
\begin{align}\label{B2}
\min_{1\leq k\leq N+1} \left\|\nabla f(x^k)\right\|\leq \left(\tfrac{4L(f(x^1)-f^\star)}{3N+2}\right)^{\tfrac{1}{2}}.
\end{align}
\end{theorem}
\begin{proof} Let $U$ denote the square of the right-side of inequality \eqref{B1} and let $B=\tfrac{U}{\Delta}$. To establish this bound, we show that $U$ is an upper bound for problem \eqref{P4}.  Consider the feasible point $\left(\{g^k; f^k\}_1^{N+1}; \ell\right)$ for problem \eqref{P4}. Suppose that
$$
\alpha_k = \tfrac{B}{2}\max\left\{2, t_kL+1\right\} \ \ \ k\in\left\{1, \ldots, N\right\}.
$$
In addition, we define $\sigma_1$ and $\sigma_k$, respectively, as follows:
\begin{align*}
& \sigma_1 =  \tfrac{B}{4}\min\left\{-Lt_1^2+3t_1, -L^2t_1^3+3t_1\right\}, \\
& \sigma_k =  \tfrac{B}{4}\min\left\{-Lt_k^2+3t_k+t_{k-1},-L^2t_k^3+3t_k+t_{k-1}\right\} \ \ k\in\{2, \ldots, N\},
\end{align*}
and $\sigma_{N+1}=1-\sum_{k=1}^N \sigma_k=\tfrac{B}{4L}(2+Lt_N)$. As $t_k\in(0, \tfrac{\sqrt{3}}{L})$ for $k\in\{1, \ldots, N\}$, the $\sigma_k$'s will be non-negative.
 It is seen that
 $$
 \sigma_k+(2\alpha_k-B)\tfrac{Lt_k^2}{4}-\tfrac{Bt_k}{2}=\tfrac{B}{4}(t_k+t_{k-1}) \ \ \ k\in\{2, \ldots, N\}.
 $$
 By using the last equality, one may verify directly through elementary  algebra that
\begin{align*}
&\ell-U+\sum_{k=1}^{N+1} \sigma_k\left( \left\| g^k\right\|^2-\ell\right)+B\left( f^\star-f^1+\Delta\right)+B\left(f^{N+1}-\tfrac{1}{2L}\|g^{N+1}\|^2 -f^\star\right)\\
& +\sum_{k=1}^{N} \alpha_k\Big(  f^{k}-f^{k+1}-\tfrac{1}{2L}\left\|g^{k}-g^{k+1}\right\|^2+\tfrac{L}{4}\left\|t_kg^{k}-\tfrac{1}{L}\left(g^{k}-g^{k+1}\right)\right\|^2\\
&-\left\langle g^{k+1},t_kg^k\right\rangle\Big)+\sum_{k=1}^{N} \left(\alpha_k-B\right)\Big(  f^{k+1}-f^{k}-\tfrac{1}{2L}\left\|g^{k+1}-g^{k}\right\|^2\\
&+\tfrac{L}{4}\left\|-t_kg^{k}-\tfrac{1}{L}\left(g^{k+1}-g^{k}\right)\right\|^2-\left\langle g^{k},-t_kg^k\right\rangle\Big)=
\tfrac{-(2\alpha_1-B)}{4L}\left\|g^1-g^{2}\right\|^2\\
&+\tfrac{Bt_1}{4}\left\|g^1\right\|^2-\tfrac{Bt_1}{2}\left\langle g^1, g^{2}\right\rangle+\tfrac{Bt_{N}}{4}\left\|g^{N+1}\right\|^2\\
&+\sum_{k=2}^N\left( \tfrac{-(2\alpha_k-B)}{4L}\left\|g^k-g^{k+1}\right\|^2+\tfrac{B(t_k+t_{k-1})}{4}\left\|g^k\right\|^2-\tfrac{Bt_k}{2}\left\langle g^k, g^{k+1}\right\rangle\right)=\\
 &-\sum_{k=1}^{N} Q_k,
\end{align*}
where
\[
Q_k = \begin{cases}
 \tfrac{B}{4}\left(\tfrac{1}{L}-t_k\right)\left\|g^k-g^{k+1}\right\|^2 &  t_k< \tfrac{1}{L}\\
 0 &  t_k\geq \tfrac{1}{L}.
\end{cases}
\]
Since $\sum_{k=1}^{N} Q_k$ is a non-negative quadratic function and the given dual multipliers are non-negative, we have $\ell\leq U$ for any feasible solution of \eqref{P4}.
\end{proof}

The special step length $t_k=\tfrac{\sqrt{4/3}}{L}$ for $k\in\left\{1, \ldots, N\right\}$ used to obtain \eqref{B3} will be motivated later in Theorem \ref{thm:optimal step size}.
Note that \eqref{B2} gives a formal proof (with a small improvement) of the bound claimed by Taylor {\cite[page 190]{Taylot_T}}; see \eqref{claim:Taylor}.

An important question concerning the bound \eqref{B1} is its difference with the optimal value of \eqref{P1}. It is known that the lower bound for Algorithm \ref{Alg1} is of the order $\Omega\left(\tfrac{1}{\sqrt{N}}\right)$  \cite{cartis2010complexity, carmon2019lower}. In what follows, we establish that the bound \eqref{B1} is exact in some cases.

\begin{proposition}
\label{prop:example}
The value
$$
\left(\tfrac{4\Delta}{\sum_{k=1}^N \min\left(-L^2t_k^3+4t_k,-Lt_k^2+4t_k\right)+\tfrac{2}{L}}\right)^{\tfrac{1}{2}}
$$
is the optimal value of \eqref{P1}  when all step lengths satisfy $t_k\in(0, \tfrac{1}{L}]$, $k\in\{1, \ldots, N\}$.
\end{proposition}
\begin{proof}
 It suffices for a given $N$ to demonstrate an $L$-smooth function $f$ and a point $x^1$ such that
\begin{align}\label{BE}
\min_{1\leq k\leq N+1} \left\|\nabla f(x^k)\right\|=  \left(\tfrac{4\Delta}{\sum_{k=1}^N \min\left(-L^2t_k^3+4t_k,-Lt_k^2+4t_k\right)+\tfrac{2}{L}}\right)^{\tfrac{1}{2}}.
\end{align}
Suppose now that $t_k\in(0, \tfrac{1}{L}]$, $k\in\{1, \ldots, N\}$, and $U$ denotes the right-hand-side of equality \eqref{BE}. We set  $t_{N+1}=\tfrac{1}{L}$. Let
 \begin{align*}
& l_i=U\left(\sum_{k=i}^{N+1} t_k \right), \ \ f^i=\Delta-\tfrac{U^2}{4}\left(\sum_{k=1}^{i-1} -Lt_k^2+4t_k\right) \ \ \ & i\in\left\{1, \ldots, N+1\right\},
\end{align*}
 and $l_{N+2}=0$. By elementary calculus, one can check that the function $f:\mathbb{R}\to\mathbb{R}$ given by
{\small{
\begin{align}\label{Ex1}
f(x) = \begin{cases}
\tfrac{L}{2}(x-l_1)^2+U(x-l_1)+f^1  &  \ \ x\in\left[\tfrac{1}{2}(l_1+l_2), \infty\right)
\vspace{1mm}
\\
\tfrac{-L}{2}(x-l_{i+1})^2+U(x-l_{i+1})+f^{i+1} &  \ \ x\in \left[l_{i+1}, \tfrac{1}{2}(l_i+l_{i+1})\right]
\vspace{1mm}\\
\tfrac{L}{2}(x-l_{i+1})^2+U(x-l_{i+1})+f^{i+1}  &  \ \ x\in \left[\tfrac{1}{2}(l_{i+1}+l_{i+2}), l_{i+1}\right]
\vspace{1mm}\\
\tfrac{L}{2}x^2 &  \ \ x\in\left(-\infty, \tfrac{1}{2}l_{N+1}\right]
\end{cases}
\end{align}
}}
 \hspace{-.22cm} for $i\in\left\{1, \ldots, N\right\}$, is $L$-smooth with the optimal value $f^\star= 0$ and the optimal solution $x^\star=0$. In addition, we have equality \eqref{BE} for $x^1=l_1$. Indeed,
 \begin{align*}
& x^i=l_i \ \ \ & i\in\left\{1, \ldots, N+1\right\}  \\
& \nabla f(x^i)=U \ \ \ & i\in\left\{1, \ldots, N+1\right\}\\
&  f(x^i)=f^i \ \ \ & i\in\left\{1, \ldots, N+1\right\}.
 \end{align*}
\end{proof}

 Figure \ref{Fig1} represents the plot of function $f$ as constructed in the proof of Proposition \ref{prop:example} for different parameters and the fixed step length $t_k=\tfrac{1}{L}$ for all $k$.

 Note that, though we have only  shown  the exactness of the bound \eqref{B1} for  step lengths in the interval $(0, \tfrac{1}{L}]$, we also conjecture that the  bound \eqref{B1} is in fact exact for all step lengths in the interval $(0, \tfrac{\sqrt{3}}{L})$.

 By minimizing the right-hand-side of \eqref{B1}, the next theorem gives the \lq optimal\rq~ step lengths with respect to the bound.
\begin{figure}
  \centering
  \begin{subfigure}[b]{0.3\textwidth}
\begin{tikzpicture}[domain=0:4]
    \draw[->] (-2,0) -- (4.5,0) node[right] {$x$};
    \draw[->] (0,-.5) -- (0,3) node[above] {$f(x)$};
    \draw[color=blue]    plot [domain = 3.4017:4, samples = 300](\x,0.5*\x^2-3.0237*\x+6.2857) ;
    \draw[color=blue]    plot [domain = 3.0237:3.4017, samples = 300](\x,-0.5*\x^2+3.7796*\x-5.2857) ;
    \draw[color=blue]    plot [domain = 2.2678:2.6458, samples = 300](\x,-0.5*\x^2+3.0237*\x-3.1429) ;
    \draw[color=blue]    plot [domain = 1.5119:1.8898, samples = 300](\x,-0.5*\x^2+2.2678*\x-1.5714) ;
    \draw[color=blue]    plot [domain = 0.7559:1.1339, samples = 300](\x,-0.5*\x^2+1.5119*\x-.5714) ;

    \draw[color=blue]    plot [domain = 2.6458:3.0237, samples = 300](\x,0.5*\x^2-2.2678*\x+3.8571) ;
    \draw[color=blue]    plot [domain = 1.8898:2.2678, samples = 300](\x,0.5*\x^2-1.5119*\x+2) ;
    \draw[color=blue]    plot [domain = 1.1339:1.5119, samples = 300](\x,0.5*\x^2-0.7559*\x+0.7143) ;
    \draw[color=blue]    plot [domain = 0.3780:0.7559, samples = 300](\x,0.5*\x^2-0*\x+4.4409e-16) ;

    \draw[color=blue]    plot [domain = 0:0.3780, samples = 300](\x,.5*\x^2) ;
    \draw[color=blue]    plot [domain = -2:0, samples = 300](\x,-.5*\x^2) ;3.0237

\draw[color=black] (3.8296,-.2) node {$x^1$};
\draw [fill=black] (3.7796,0) circle (1.1pt);

\draw[dotted,color=red] (3.0237,0) -- (3.0237,3);

\draw[color=black] (3.0737,-.2) node {$x^2$};
\draw [fill=black] (3.0237,0) circle (1.1pt);
\draw[dotted,color=red] (3.4017,0) -- (3.4017,3);


\draw[dotted,color=red] (1.8898,0) -- (1.8898,3);

\draw[dotted,color=red] (2.6458,0) -- (2.6458,3);
\draw[color=black] (2.3178,-.2) node {$x^3$};
\draw [fill=black] (2.2678,0) circle (1.1pt);

\draw[dotted,color=red] (2.2678,0) -- (2.2678,3);

\draw[dotted,color=red] (1.1339,0) -- (1.1339,3);
\draw[color=black] (1.5619,-.2) node {$x^4$};
\draw [fill=black] (1.5119,0) circle (1.1pt);

\draw[dotted,color=red] (1.5119,0) -- (1.5119,3);

\draw[dotted,color=red] (0.7559,0) -- (0.7559,3);

\draw[dotted,color=red] (0.3780,0) -- (0.3780,3);
\draw[color=black] (0.8059,-.2) node {$x^5$};
\draw [fill=black] (0.7559,0) circle (1.1pt);

\end{tikzpicture}
\caption{$N=4,\ \Delta=2,\ L=1$} \label{fig:M1}
\end{subfigure}

\begin{subfigure}[b]{0.3\textwidth}
\begin{tikzpicture}[domain=0:4]
    \draw[->] (-2,0) -- (4.4,0) node[right] {$x$};
    \draw[->] (0,-.5) -- (0,4.4) node[above] {$f(x)$};

    \draw[color=blue]    plot [domain = 2.9848:3.6, samples = 300](\x,\x^2-5.1168*\x+9.8182) ;

    \draw[color=blue]    plot [domain = 2.5584:2.9848, samples = 300](\x,-\x^2+6.8224*\x-8) ;
    \draw[color=blue]    plot [domain = 1.7056:2.1320, samples = 300](\x,-\x^2+5.1168*\x-4) ;
    \draw[color=blue]    plot [domain = 0.8528:1.2792, samples = 300](\x,-\x^2+3.41128*\x-1.4545) ;

    \draw[color=blue]    plot [domain = 2.1320:2.5584, samples = 300](\x,\x^2-3.4112*\x+5.0909) ;
    \draw[color=blue]    plot [domain = 1.2792:1.7056, samples = 300](\x,\x^2-1.7056*\x+1.8182) ;
    \draw[color=blue]    plot [domain = 0.4264:0.8528, samples = 300](\x,\x^2-0*\x+0) ;

    \draw[color=blue]    plot [domain = 0:0.4264, samples = 300](\x,\x^2) ;
    \draw[color=blue]    plot [domain = -2.05:0, samples = 300](\x,-\x^2) ;

\draw[dotted,color=red] (2.9848,0) -- (2.9848,4.4);
\draw[color=black] (3.4612,-.2) node {$x^1$};
\draw [fill=black] (3.4112,0) circle (1.1pt);

\draw[dotted,color=red] (2.5584,0) -- (2.5584,4.4);

\draw[dotted,color=red] (2.1320,0) -- (2.1320,4.4);
\draw[color=black] (2.6084,-.2) node {$x^2$};
\draw [fill=black] (2.5584,0) circle (1.1pt);

\draw[dotted,color=red] (1.7056,0) -- (1.7056,4.4);

\draw[dotted,color=red] (1.2792,0) -- (1.2792,4.4);
\draw[color=black] (1.7556,-.2) node {$x^3$};
\draw [fill=black] (1.7056,0) circle (1.1pt);

\draw[dotted,color=red] (0.8528,0) -- (0.8528,4.4);

\draw[dotted,color=red] (0.4264,0) -- (0.4264,4.4);
\draw[color=black] (0.9028,-.2) node {$x^4$};
\draw [fill=black] (0.8528,0) circle (1.1pt);

\end{tikzpicture}
\caption{$N=3,\ \Delta=4,\ L=2$} \label{fig:M2}
\end{subfigure}
\caption{ Plot of the function $f$ in \eqref{Ex1} for different parameters and $t_k=\tfrac{1}{L}$. (Dotted lines denote the endpoints of intervals.)} \label{Fig1}
\end{figure}
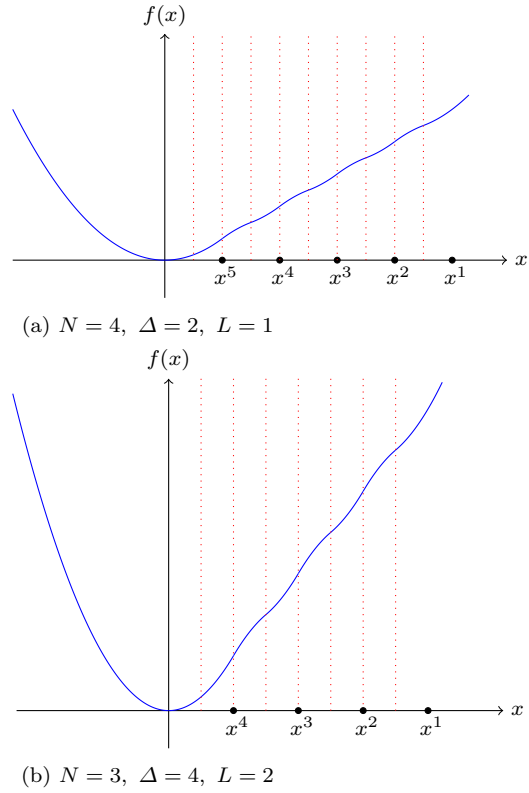
\begin{theorem}
\label{thm:optimal step size}
Let $f$ be an $L$-smooth function. Then the optimal step size for gradient method with respect to  bound \eqref{B1} is given by
\[
t_k=\tfrac{\sqrt{\tfrac{4}{3}}}{L}\ \ \forall k\in\left\{1,\dots,N\right\},
\]
provided that $t_k \in (0, \tfrac{\sqrt{3}}{L})$ for all $k \in \{1,\ldots,N\}$.
\end{theorem}
\begin{proof}
We minimize the right-hand-side of \eqref{B1}, that is
\[
\min_{t_k\in(0,\tfrac{\sqrt{3}}{L})} \ \left(\tfrac{4\Delta}{\sum_{k=1}^N \min\left(-L^2t_k^3+4t_k,-Lt_k^2+4t_k\right)+\tfrac{2}{L}}\right)^{\tfrac{1}{2}},
\]
which is equivalent to maximizing
\[
\max_{t\in\left(0,\tfrac{\sqrt{3}}{L}\right)^N} \ H(t):= \sum_{k=1}^N \min\left(-L^2t_k^3+4t_k,-Lt_k^2+4t_k\right).
\]
Since $H$ is a strictly concave function on $\left(0,\tfrac{\sqrt{3}}{L}\right)^N$ and at $\bar t$ given by
\[
\bar t_k=\tfrac{\sqrt{\tfrac{4}{3}}}{L}\ \ \forall k\in\left\{1,\dots,N\right\},
\]
we have $\nabla H\left(\bar t\right)=0$, which shows that $\bar t$ is the unique maximum solution of $H$ over $\left(0,\tfrac{\sqrt{3}}{L}\right)^N$.
\end{proof}

The step length $\tfrac{1}{L}$ commonly is regarded as the optimal step length  in the literature; see  \cite[Chapter 1]{Nesterov}.  Due to the
example introduced in \eqref{Ex1}, we see that  the worst-case convergence rate for the step length $\tfrac{1}{L}$
 cannot be better than $\left(\tfrac{4L(f(x^1)-f^\star)}{3N+2}\right)^{\tfrac{1}{2}}$. By our analysis,
 it follows that,  for the step length $\tfrac{2\sqrt{3}}{3L}$, we get the convergence rate \eqref{B3},
 which is better than $\left(\tfrac{4L(f(x^1)-f^\star)}{3N+2}\right)^{\tfrac{1}{2}}$, since the constant in the bound improves from ca.\ $\tfrac{4}{3}\approx 1.333$ to $\tfrac{6\sqrt{3}}{8} \approx 1.299$.
\section{Concluding remarks}
In this paper, we studied the convergence rate of gradient  method for $L$-smooth functions
 and we provided a new convergence rate when the step lengths belong to the interval
  $(0, \tfrac{\sqrt{3}}{L})$. Moreover, we have shown that this convergence rate is tight for step lengths in the interval $(0,\tfrac{1}{L}]$. As mentioned in the introduction,  Algorithm \ref{Alg1} is convergent for the step lengths
   in the larger interval  $(0, \tfrac{2}{L})$. Following extensive numerical experiments,
   where we solved the semidefinite program \eqref{S1} for different parameter values,
   we conjecture that when  $t_k\in\left(0, \tfrac{2}{L}\right)$ for $k\in\left\{1, \ldots, N\right\}$, we have
\begin{align*}
 \min_{1\leq k\leq N+1}\left\|\nabla f(x^k)\right\|\leq \left(\frac{4\Delta}{\sum_{k=1}^N \min\left(-L^2t_k^3+4t_k,-Lt_k^2+4t_k\right)}\right)^{\tfrac{1}{2}},
\end{align*}
under the same conditions as for Theorem \ref{thm:main}. The right-hand-side is again minimized by the constant step length
$t_k=\tfrac{\sqrt{4/3}}{L}$ for all $k \in \{1,\ldots,N\}$.

\begin{acknowledgements}
The authors would like to thank Adrien Taylor for valuable discussions. We are also very grateful to two anonymous referees for their valuable comments and suggestions which help to improve the paper considerably. In particular, one of the reviewers pointed out the result in \cite[Theorem 7 in Appendix]{Drori} to us, which was most helpful.
\end{acknowledgements}

%
%

\bibliographystyle{spmpsci}      
\bibliography{references}

\end{document}